\documentclass[11pt,a4paper]{article}
\usepackage[T1]{fontenc}
\usepackage[margin=1in]{geometry}
\usepackage{amsmath,amssymb,amsthm,mathtools}
\usepackage{enumitem}
\usepackage{microtype}
\usepackage[numbers,sort&compress]{natbib}
\usepackage[hidelinks]{hyperref}
\hypersetup{pdftitle={Proper conflict-free choosability of sparse graphs with girth at least seven},
  pdfauthor={Xingqin Qi, Huimin Song, and Zhulou Cao}}

\newtheorem{theorem}{Theorem}[section]
\newtheorem{lemma}[theorem]{Lemma}
\newtheorem{proposition}[theorem]{Proposition}
\newtheorem{corollary}[theorem]{Corollary}
\newtheorem{conjecture}[theorem]{Conjecture}
\theoremstyle{definition}

\newcommand{\mad}{\operatorname{mad}}
\newcommand{\UU}{\mathcal U}
\newcommand{\col}{\varphi}

\begin{document}
\title{Proper conflict-free choosability of sparse graphs with girth at least seven}
\author{Xingqin Qi\thanks{Email: \href{mailto:qixingqin@sdu.edu.cn}{\texttt{qixingqin@sdu.edu.cn}}.}
  \quad Huimin Song\thanks{Email: \href{mailto:hmsong@sdu.edu.cn}{\texttt{hmsong@sdu.edu.cn}}.}
  \quad Zhulou Cao\thanks{Corresponding author. Email: \href{mailto:zlouc@sdu.edu.cn}{\texttt{zlouc@sdu.edu.cn}}.}\\[0.5em]
  \small School of Mathematics and Statistics, Shandong University\\
  \small Weihai 264209, China}
\date{September 27, 2026}
\maketitle

\begin{abstract}
A proper conflict-free coloring of a graph is a proper vertex coloring in which every non-isolated vertex has a color appearing exactly once in its neighborhood. A graph $G$ is proper conflict-free $(\mathrm{degree}+2)$-choosable if every list assignment $L$ with $|L(v)|\ge d_G(v)+2$ for each $v\in V(G)$ admits such a coloring from the lists. We prove that every graph with girth at least $7$ and maximum average degree less than $8/3$ is proper conflict-free $(\mathrm{degree}+2)$-choosable. Consequently, every planar graph of girth at least $8$ has this property, improving the previously established sufficient girth bound of $9$.
\end{abstract}



\section{Introduction}
All graphs considered in this paper are finite, simple, and undirected. For a graph $G$ and a vertex $v\in V(G)$, let $N_G(v)$ be the set of neighbors of $v$ and let $d_G(v)=|N_G(v)|$. A vertex coloring $\varphi$ of $G$ is \emph{proper} if $\varphi(u)\ne\varphi(v)$ whenever $uv\in E(G)$. A \emph{proper conflict-free coloring} (PCF coloring for short) is a proper coloring such that, for every non-isolated vertex $v$, some color appears exactly once in $N_G(v)$. A PCF $k$-coloring uses at most $k$ colors.

A \emph{list assignment} $L$ assigns a set $L(v)$ of available colors to each vertex $v$. An \emph{$L$-coloring} is a coloring $\varphi$ satisfying $\varphi(v)\in L(v)$ for every $v$. Given a nonnegative integer-valued function $f$ on $V(G)$, we say that $G$ is \emph{proper conflict-free $f$-choosable} if every list assignment $L$ with $|L(v)|\ge f(v)$ admits a PCF $L$-coloring. In particular, for an integer $k\ge0$, the choice $f(v)=d_G(v)+k$ defines \emph{proper conflict-free $(\mathrm{degree}+k)$-choosability}.

The \emph{girth} $g(G)$ is the length of a shortest cycle in $G$; we set $g(G)=\infty$ if $G$ is a forest. The \emph{maximum average degree} of $G$ is
\[
   \mad(G):=\max_{\varnothing\ne H\subseteq G}\frac{2|E(H)|}{|V(H)|},
\]
where the maximum is taken over all nonempty subgraphs of $G$. We set $\mad(\varnothing)=0$.

Conflict-free coloring was introduced for geometric set systems by Even et al.~\cite{EvenEtAl2003}, motivated by frequency assignment. Pach and Tardos~\cite{PachTardos2009} studied the condition for general hypergraphs and graph neighborhoods. Fabrici, Lu\v{z}ar, Rindo\v{s}ov\'a, and Sot\'ak~\cite{FabriciEtAl2023} introduced the proper version for graph neighborhoods. They proved that every planar graph has a PCF $8$-coloring. A related notion is an \emph{odd coloring}: a proper coloring in which every non-isolated vertex has a color appearing an odd number of times in its neighborhood~\cite{PetrusevskiSkrekovski2022}. Every PCF coloring is an odd coloring.

Caro, Petru\v{s}evski, and \v{S}krekovski~\cite{CaroPetrusevskiSkrekovski2023} obtained PCF coloring bounds in terms of maximum degree and maximum average degree; in particular, $\mad(G)<8/3$ guarantees a PCF $6$-coloring. Cho, Choi, Kwon, and Park~\cite{ChoChoiKwonPark2025} obtained sharp maximum-average-degree bounds for fixed numbers of colors and proved PCF $7$-colorability for planar graphs of girth at least $5$. Anderson et al.~\cite{AndersonEtAl2025} developed the forb-flex method for odd and PCF coloring of planar graphs. More recently, Jim\'enez, Lintzmayer, and Sambinelli~\cite{JimenezLintzmayerSambinelli2026} proved in a preprint that every planar graph has a PCF $7$-coloring.

Structural bounds for PCF coloring and its list version were obtained by Hickingbotham~\cite{Hickingbotham2023} and Liu~\cite{Liu2024}, respectively. Cranston and Liu~\cite{CranstonLiu2024} and Liu and Reed~\cite{LiuReed2025} studied asymptotic bounds in terms of maximum degree. In these results, the number of available colors is controlled by a parameter of the whole graph. Here we consider list sizes that depend on the degree of each vertex.

Kashima, \v{S}krekovski, and Xu~\cite{KashimaSkrekovskiXuLists} studied degree-dependent list sizes and proposed the following conjecture.
\begin{conjecture}[\cite{KashimaSkrekovskiXuLists}]\label{conj:degree-two}
Every connected graph other than $C_5$ is proper conflict-free $(\mathrm{degree}+2)$-choosable.
\end{conjecture}
They proved the conjecture for connected graphs of maximum degree at most $3$~\cite{KashimaSkrekovskiXuLists} and connected outerplanar graphs~\cite{KashimaSkrekovskiXuOuterplanar}, in both cases excluding $C_5$. A graph is \emph{$d$-degenerate} if every nonempty subgraph has a vertex of degree at most $d$. They also proved that every $d$-degenerate graph is proper conflict-free $(\mathrm{degree}+d+1)$-choosable and that every tree is proper conflict-free $(\mathrm{degree}+1)$-choosable~\cite{KashimaSkrekovskiXuDegeneracy}. Lu, Ying, and Song~\cite{LuYingSong2026} obtained the stronger list bound of four colors at $2$-vertices and $d_G(v)+1$ colors at all other vertices for connected graphs other than $C_5$ with maximum degree at most $3$ or maximum average degree less than $12/5$.

For planar graphs, Wang and Zhang~\cite{WangZhang2025} proved proper conflict-free $(\mathrm{degree}+2)$-choosability when the girth is at least $12$ and asked whether girth at least $6$ suffices~\cite[Problem~5.4]{WangZhang2025}. Kashima, \v{S}krekovski, and Xu~\cite{KashimaSkrekovskiXu2026} proved that every connected graph $G\ne C_5$ with $\mad(G)<18/7$ is proper conflict-free $(\mathrm{degree}+2)$-choosable. This implies the planar case of girth at least $9$. We prove the following.

\begin{theorem}\label{thm:main}
Let $G$ be a finite simple graph. If
\[
   g(G)\ge 7
   \qquad\text{and}\qquad
   \mad(G)<\frac83,
\]
then $G$ is proper conflict-free $(\mathrm{degree}+2)$-choosable.
\end{theorem}

The maximum-average-degree bound in Theorem~\ref{thm:main} is larger than $18/7$, but the theorem also assumes girth at least $7$. For a planar graph of girth at least $8$, Euler's formula gives $|E(H)|\le\frac43(|V(H)|-2)$ for every connected subgraph $H$ containing a cycle, while every tree has average degree less than $2$. Summing over the components of an arbitrary nonempty subgraph gives $\mad(G)<8/3$. Thus Theorem~\ref{thm:main} yields the following improvement of the planar girth bound from $9$ to $8$.

\begin{corollary}\label{cor:planar}
Every planar graph of girth at least $8$ is proper conflict-free $(\mathrm{degree}+2)$-choosable.
\end{corollary}

The proof uses reducible configurations and discharging.

\section{Preliminaries}
We omit the subscript $G$ from $N_G(v)$ and $d_G(v)$ when the graph is clear. A $k$-vertex (respectively, a $k^+$-vertex) has degree $k$ (respectively, at least $k$); a $k$-neighbor is a neighbor of degree $k$.

For a subgraph $H$ of $G$, a coloring $\col$ of $H$, and a vertex $x\in V(H)$, let
\[
  \UU_\col(x,H):=\left\{c:\bigl|\{y\in N_H(x):\col(y)=c\}\bigr|=1\right\}.
\]
A color in $\UU_\col(x,H)$ is called a \emph{conflict-free color} of $x$, or a \emph{witness color} at $x$. Thus $\UU_\col(x,H)\ne\varnothing$ whenever $\col$ is PCF and $x$ is non-isolated in $H$. We use $\col$ also for an extension of the coloring to additional vertices.

Following~\cite{KashimaSkrekovskiXu2026}, a \emph{$t$-thread} is a path on $t$ vertices of degree $2$ in $G$. It is \emph{maximal} if it is not contained in a longer such path. The vertices adjacent to its ends outside the thread are its \emph{boundary vertices}. A thread is \emph{incident with} a vertex $x$ if $x$ is one of its boundary vertices. We write
\[
  x-v_1-\cdots-v_t-y
\]
for the thread together with its boundary vertices. Here $t$ counts the $2$-vertices $v_1,\ldots,v_t$; when $x\ne y$, the displayed path has $t+1$ edges. The definition does not require the two boundary vertices to be distinct. When extending a coloring over a thread incident with a specified vertex $x$, we call $x$ the \emph{root}, $v_1$ its \emph{first vertex}, and $y$ its \emph{other endpoint}. We also use \emph{branch} for a thread viewed from its root.

\begin{lemma}\label{lem:three-neighbors}
If a vertex has three neighbors and at least two colors occur on those three neighbors, then one of the colors occurs exactly once.
\end{lemma}
\begin{proof}
If no color occurred exactly once, each of the at least two colors would occur at least twice, requiring at least four neighbors.
\end{proof}

\subsection{A minimum counterexample}\label{sec:minimal}

Assume, for contradiction, that Theorem~\ref{thm:main} is false.
Since a graph is proper conflict-free list-colorable if and only if each
of its components is, we may choose a connected counterexample $G$ with
$|V(G)|$ minimum.  Fix a list assignment $L$ witnessing the failure.
By deleting excess colors from lists, we assume throughout that
\[
  |L(v)|=d_G(v)+2 \qquad (v\in V(G)).
\]
For every proper induced subgraph $H$ of $G$, we have $g(H)\ge7$, $\mad(H)<8/3$, and $|L(x)|=d_G(x)+2\ge d_H(x)+2$ for all $x\in V(H)$. Hence minimality gives a PCF coloring of $H$ from these lists. We also apply minimality to other list assignments on $H$ when their sizes are at least $d_H(x)+2$.

It is known that every cycle distinct from $C_5$ is proper conflict-free
$4$-choosable; see, for example, \cite{KashimaSkrekovskiXu2026}.
Consequently $G$ is not a cycle.

\section{Local extension lemmas}\label{sec:branch}
Let $H$ be the graph obtained by deleting the root $v$ and the internal vertices of the threads to be colored, and let $\col$ be a PCF $L$-coloring of $H$. Each deleted $2$-vertex has a list of at least four colors. The following lemmas extend a prescribed color $\alpha$ at $v$ over these threads. They ensure properness on the restored edges and the conflict-free condition at the internal and boundary vertices; the conditions at $v$ and its other neighbors are checked in the applications.

For each non-isolated boundary vertex $x$ of $H$, fix $c_x\in\UU_\col(x,H)$. Such a color exists, and properness gives $c_x\ne\col(x)$. When several threads end at $x$, we use the same $c_x$ on all of them. Requiring every new neighbor of $x$ to avoid $c_x$ preserves this witness.

\begin{lemma}[A $1$-thread ending at a $4^+$-vertex]\label{lem:one-four}
Consider a path $vux$ with $d_G(u)=2$ and $d_G(x)\ge4$. Suppose $x$ is non-isolated in $H$, and fix $c_x\in\UU_\col(x,H)$. If $\alpha\ne\col(x)$, there is a color for $u$ such that the edges $vu,ux$ are properly colored, $u$ is conflict-free, and $c_x$ remains a witness at $x$.
\end{lemma}
\begin{proof}
Choose
\[
  \col(u)\in L(u)\setminus\{\alpha,\col(x),c_x\}.
\]
Such a color exists since $|L(u)|\ge4$. It differs from the colors on $v$ and $x$. These two colors are distinct, so $u$ is conflict-free. Since $\col(u)\ne c_x$, the witness at $x$ is preserved.
\end{proof}

\begin{lemma}[A $1$-thread ending at a $3$-vertex]\label{lem:one-three}
Consider a path $vux$ with $d_G(u)=2$ and $d_G(x)=3$. Suppose both neighbors of $x$ other than $u$ belong to $H$. If $\alpha\ne\col(x)$, then every color in $L(u)\setminus\{\alpha,\col(x)\}$ properly colors the edges $vu,ux$ and makes $u$ and $x$ conflict-free. In particular, there are at least two such colors.
\end{lemma}
\begin{proof}
The two neighbors of $x$ in $H$ have distinct colors, since $\col$ is conflict-free at $x$. Choose
\[
  \col(u)\in L(u)\setminus\{\alpha,\col(x)\}.
\]
At least two colors are available. Both incident edges of $u$ are properly colored, and the colors on its two neighbors are distinct. The three neighbors of $x$ still use at least two colors, so Lemma~\ref{lem:three-neighbors} applies at $x$.
\end{proof}

\begin{lemma}[Control set for a $2$-thread]\label{lem:two-thread}
Consider a path $v-u-w-x$ with $d_G(u)=d_G(w)=2$, and suppose $x$ is non-isolated in $H$. Put $p=\col(x)$ and fix $c_x\in\UU_\col(x,H)$. Choose a two-element set
\[
  B\subseteq L(w)\setminus\{p,c_x\}.
\]
For every root color $\alpha$, there is a set $F_\alpha\subseteq L(u)$ such that
\[
  \alpha\notin B \Longrightarrow |F_\alpha|\ge2,
  \qquad
  \alpha\in B \Longrightarrow |F_\alpha|\ge1.
\]
For every $a\in F_\alpha$, assigning $a$ to $u$ extends to a coloring of $w$ that is proper on the path, is conflict-free at $u,w$, and preserves $c_x$ at $x$.
\end{lemma}
\begin{proof}
The set $B$ exists because $|L(w)|\ge4$ and at most two colors are excluded. If $\alpha\notin B$, let $F_\alpha=L(u)\setminus\{\alpha,p\}$. Then $|F_\alpha|\ge2$. For $a\in F_\alpha$, set $\col(u)=a$ and choose $\col(w)\in B\setminus\{a,\alpha\}$.

If $\alpha\in B$, write $B=\{\alpha,\beta\}$ and let
\[
  F_\alpha=L(u)\setminus\{\alpha,p,\beta\}.
\]
Then $|F_\alpha|\ge1$. For $a\in F_\alpha$, set $\col(u)=a$ and $\col(w)=\beta$.

In both cases, $\col(w)$ differs from $p,c_x,a,\alpha$, and $a$ differs from $\alpha,p$. Thus the path is properly colored. The neighbors of $u$ have distinct colors $\alpha,\col(w)$, and the neighbors of $w$ have distinct colors $a,p$. Finally, $\col(w)\ne c_x$ preserves the witness at $x$.
\end{proof}

We call $B$ a \emph{control set}. Relative to a root color $\alpha$, the branch is \emph{flexible} if $\alpha\notin B$ and \emph{rigid} if $\alpha\in B$. These terms refer to the guaranteed number of choices in the lemma: at least two for a flexible branch and at least one for a rigid branch.

\begin{lemma}[Control set for a $3$-thread]\label{lem:three-thread}
Consider a path $v-u-z-w-x$ with $d_G(u)=d_G(z)=d_G(w)=2$, and suppose $x$ is non-isolated in $H$. Put $p=\col(x)$ and fix $c_x\in\UU_\col(x,H)$. There is a two-element set $C\subseteq L(u)$ such that, for every root color $\alpha$, each choice $\col(u)\in C\setminus\{\alpha\}$ extends over $z,w$, properly coloring the path, making $u,z,w$ conflict-free, and preserving $c_x$ at $x$.
\end{lemma}
\begin{proof}
Since $|L(w)|\ge4$, choose a two-element set $W\subseteq L(w)\setminus\{p,c_x\}$. Since $|L(u)|\ge4$ and $|W|=2$, choose a two-element set $C\subseteq L(u)\setminus W$. Assign to $u$ any color in $C\setminus\{\alpha\}$. Next choose
\[
  \col(z)\in L(z)\setminus\{\alpha,\col(u),p\},
\]
and then choose $\col(w)\in W\setminus\{\col(z)\}$. Both choices are possible. Since $C\cap W=\varnothing$, we have $\col(u)\ne\col(w)$. The neighbors of $u,z,w$ therefore have distinct colors in each case. All path edges are properly colored, and $\col(w)\ne c_x$ preserves the witness at $x$.
\end{proof}

The set $C$ is the control set of the $3$-thread. As for a $2$-thread, we call the branch flexible if $\alpha\notin C$ and rigid if $\alpha\in C$. The set $C\setminus\{\alpha\}$ then has two colors or one color, respectively.

\begin{lemma}[Simultaneous extension at a shared boundary]\label{lem:shared-boundary}
Let $\mathcal B$ be a family of maximal $2$- and $3$-threads incident with a root $v$, and let $H$ be an induced subgraph obtained by deleting $v$ and their internal vertices, possibly together with other vertices. Suppose $H$ has a PCF $L$-coloring $\col$ and every other endpoint $x$ of these threads is non-isolated in $H$. Choose one witness $c_x\in\UU_\col(x,H)$ for each distinct endpoint $x$. Assign a color $\alpha$ to $v$, and on each thread choose a color for its neighbor of $v$ as in Lemma~\ref{lem:two-thread} or Lemma~\ref{lem:three-thread}, using the same $c_x$ for all threads ending at $x$. Then all these threads can be completed simultaneously, with every path edge properly colored, every internal $2$-vertex conflict-free, and every chosen witness preserved at its endpoint.
\end{lemma}
\begin{proof}
The internal vertices of different threads are disjoint and have degree $2$, so there are no edges between the interiors of different threads. Complete each thread by the corresponding lemma. Every restored neighbor of an endpoint $x$ receives a color different from both $\col(x)$ and $c_x$. Hence all incident thread edges at $x$ are properly colored, and the single occurrence of $c_x$ in $N_H(x)$ remains unique. All conditions at internal vertices involve only the colors on their own thread, so the extensions can be made simultaneously.
\end{proof}

\begin{lemma}[Singleton selection]\label{lem:singleton}
Suppose the colors of at most three neighbors of a vertex $v$ have been fixed, and at least one neighbor remains uncolored. If each uncolored neighbor has at least two available colors, then colors can be chosen for them so that some color occurs exactly once in $N(v)$.
\end{lemma}
\begin{proof}
If a color $\gamma$ already occurs exactly once among the fixed colors, make every uncolored neighbor avoid $\gamma$. Otherwise, if any colors are fixed, all fixed colors must be equal, say to $\gamma$, since at most three neighbors are colored. Choose a color $\delta\ne\gamma$ for one uncolored neighbor and make the others avoid $\delta$. If no colors are fixed, choose any color $\delta$ for one neighbor and make the others avoid $\delta$. Each choice is possible because every uncolored neighbor has at least two available colors.
\end{proof}

\section{Structural reductions}\label{sec:reductions}

\begin{proposition}[Structural reductions]\label{prop:reductions}
The minimum counterexample $G$ satisfies all of the following.
\begin{enumerate}[label=\textnormal{(R\arabic*)}]
\item\label{red:mindeg} $\delta(G)\ge2$.
\item\label{red:3no2} No $3$-vertex is incident with a $2$-thread.
\item\label{red:no4} $G$ has no $4$-thread.
\item\label{red:3plus} Every $3^+$-vertex has a $3^+$-neighbor.
\item\label{red:n3} If $v$ is a $k$-vertex with $k\ge4$, then $v$ is
incident with at most $k-3$ maximal $3$-threads.
\item\label{red:saturated} If a $k$-vertex $v$ with $k\ge4$ has exactly $k-1$
$2$-neighbors, then at most $k-4$ of those neighbors lie on maximal
$3$-threads.
\item\label{red:3two} If a $3$-vertex has two $2$-neighbors, then the
other endpoint of each corresponding maximal $1$-thread is a
$4^+$-vertex.
\end{enumerate}
\end{proposition}

We prove (R1)--(R7) in order.

\begin{proof}[Proof of Proposition~\ref{prop:reductions}]
\textbf{(R1).}
Suppose first that $u$ is a $1$-vertex with neighbor $x$.
If $G=K_2$, choose distinct colors on its two vertices and we are done.
Otherwise $x$ has a neighbor other than $u$ (since $G$ is connected), so $x$ is non-isolated in $G-u$.  Color $G-u$ by minimality and fix
$c_x\in\UU_\col(x,G-u)$.  Since $|L(u)|=3$, choose
\[
  \col(u)\in L(u)\setminus\{\col(x),c_x\}.
\]
The old witness at $x$ is preserved, while the unique neighbor of $u$
is itself a conflict-free witness for $u$.  An isolated vertex is even
more immediate.  Thus $\delta(G)\ge2$.

\medskip
\noindent\textbf{(R2).}
Suppose a $3$-vertex $u$ is incident with a $2$-thread
\[
  u-v_1-v_2-x,
\]
where $v_1,v_2$ are $2$-vertices.  Color
$H=G-\{v_1,v_2\}$ by minimality.  The vertex $x$ is non-isolated in
$H$, since it loses only $v_2$ and $d_G(x)\ge2$; fix
$c_x\in\UU_\col(x,H)$.  Choose
\[
 \col(v_2)\in L(v_2)\setminus
 \{\col(u),\col(x),c_x\}
\]
and then
\[
 \col(v_1)\in L(v_1)\setminus
 \{\col(v_2),\col(u),\col(x)\}.
\]
Both choices exist since the two lists have size $4$.
The vertices $v_1,v_2$ are conflict-free because their two neighbors
receive different colors.  The witness $c_x$ is preserved at $x$.
Finally, in $H$ the vertex $u$ has degree $2$, so its two neighbors in
$H$ have distinct colors; after $v_1$ is colored and the edge $uv_1$ is restored, at least two colors
still occur in $N_G(u)$, and Lemma~\ref{lem:three-neighbors} applies.
This is a contradiction.

\medskip
\noindent\textbf{(R3).}
Suppose $v_1v_2v_3v_4$ is a $4$-thread, with outside neighbors $x$ of
$v_1$ and $y$ of $v_4$.  The equality $x=y$ would create a $5$-cycle,
so $x\ne y$.  Color
$H=G-\{v_1,v_2,v_3,v_4\}$ by minimality.  Both $x$ and $y$ are
non-isolated, because each loses only one neighbor and $\delta(G)\ge2$.
Fix witnesses $c_x$ and $c_y$.
Set
\begin{align*}
 A_1&=L(v_1)\setminus\{\col(x),c_x\},&
 A_2&=L(v_2)\setminus\{\col(x)\},\\
 A_3&=L(v_3)\setminus\{\col(y)\},&
 A_4&=L(v_4)\setminus\{\col(y),c_y\}.
\end{align*}
Then $|A_1|,|A_4|\ge2$ and $|A_2|,|A_3|\ge3$.
Shrink $A_1$ to two colors.  Choose successively
\[
 \col(v_2)\in A_2\setminus A_1,
 \qquad
 \col(v_4)\in A_4\setminus\{\col(v_2)\},
\]
\[
 \col(v_3)\in A_3\setminus
    \{\col(v_2),\col(v_4)\},
 \qquad
 \col(v_1)\in A_1\setminus\{\col(v_3)\}.
\]
The coloring is proper.  Moreover, the two neighbors of each $v_i$
have distinct colors: for $v_1$ this follows from $\col(v_2)\ne\col(x)$;
for $v_2$ from $\col(v_1)\ne\col(v_3)$; for $v_3$ from
$\col(v_2)\ne\col(v_4)$; and for $v_4$ from
$\col(v_3)\ne\col(y)$.  The witnesses at $x,y$ are preserved.
Thus the coloring extends to $G$, a contradiction.

By (R1), connectedness, and the fact that $G$ is not a cycle, every maximal thread now has boundary vertices of degree at least $3$. By (R3), it has at most three $2$-vertices. Its two boundary vertices must be distinct, since otherwise it would form a cycle of length at most $4$.

\medskip
\noindent\textbf{(R4).}
Suppose that a $k$-vertex $u_0$, $k\ge3$, has only $2$-neighbors
$u_1,\ldots,u_k$.  For each $i$, let $x_i$ be the other neighbor of
$u_i$.  Girth at least $7$ implies that the vertices $x_i$ are pairwise
distinct and that no $x_i$ is one of the $u_j$: otherwise there is a
$3$- or $4$-cycle.

Since $d_G(u_i)=2$, we have $|L(u_i)|=4$ for every $i$.
For $c\in L(u_0)$ put
\[
 U_c=\{u_i:c\in L(u_i)\}.
\]
There is a color $\alpha\in L(u_0)$ with
\[
 |U_\alpha|\le\min\{3,k-1\}.
\]
Indeed, if $k=3$, choose
$\alpha\in L(u_0)\setminus L(u_1)$; if $k\ge4$, then
\[
 \sum_{c\in L(u_0)}|U_c|\le4k
\]
and hence some $\alpha$ satisfies
\[
 |U_\alpha|\le
 \left\lfloor\frac{4k}{k+2}\right\rfloor\le3.
\]

Let $H=G-\{u_0,u_1,\ldots,u_k\}$.  Each $x_i$ is non-isolated in
$H$: it has degree at least $2$ in $G$, and the only deleted neighbor it
can have is $u_i$.  Indeed, an edge $x_i u_0$ would create a triangle,
and an edge $x_i u_j$ with $j\ne i$ would create a $4$-cycle.  Define lists on $H$ by
\[
 L'(x_i)=L(x_i)\setminus\{\alpha\}
\]
for every $i$, and $L'(z)=L(z)$ otherwise.  Since each $x_i$ loses one
neighbor,
\[
 |L'(x_i)|\ge d_H(x_i)+2,
\]
so minimality gives a proper conflict-free $L'$-coloring $\col$ of
$H$.  Fix $c_i\in\UU_\col(x_i,H)$.

Set $\col(u_0)=\alpha$. For each $u_i\in U_\alpha$, choose
\[
  \col(u_i)\in L(u_i)\setminus\{\col(x_i),c_i,\alpha\}.
\]
There are at most three such vertices. For every remaining $u_i$, the set $L(u_i)\setminus\{\col(x_i),c_i\}$ contains at least two colors, none equal to $\alpha$. Since $|U_\alpha|\le k-1$, at least one neighbor of $u_0$ remains uncolored. Lemma~\ref{lem:singleton} therefore allows us to color the remaining neighbors so that $u_0$ is conflict-free.

Each $u_i$ receives a color different from $\alpha$, $\col(x_i)$, and $c_i$. Moreover, $\col(x_i)\ne\alpha$ by the definition of $L'$. Thus the extension is proper, every $u_i$ is conflict-free, and each witness $c_i$ is preserved. This gives a PCF $L$-coloring of $G$, a contradiction.

\medskip
\noindent\textbf{(R5).}
Suppose a $k$-vertex $u_0$, $k\ge4$, is incident with at least $k-2$ maximal
$3$-threads.  Select $k-2$ of them and write these as
\[
 u_0-u_i-v_i-w_i-x_i,
 \qquad 1\le i\le k-2,
\]
and let $a,b$ be the two remaining neighbors of $u_0$. Set
\[
 H=G-\left(\{u_0\}\cup
       \bigcup_{i=1}^{k-2}\{u_i,v_i,w_i\}\right).
\]
Any unselected threads are retained in $H$. By (R2)--(R3), each $x_i$ is a $4^+$-vertex. By (R4), it has a $3^+$-neighbor other than $u_0$, since $x_i u_0$ would create a $5$-cycle. This neighbor belongs to $H$, so $x_i$ is non-isolated in $H$.

The girth condition excludes $a,b$ from the selected threads and their other endpoints. It also excludes edges from $a$ or $b$ to the deleted internal vertices: such an edge would give a cycle of length at most $6$. Thus $a$ and $b$ each lose only $u_0$ and remain non-isolated by (R1), whether or not they have degree $2$ in $G$. Color $H$ by minimality, and fix witnesses at $a,b$ and at each distinct endpoint $x_i$. Put
\[
 R=L(u_0)\setminus
 \{\col(a),c_a,\col(b),c_b\}.
\]
Thus $|R|\ge k-2$. For each selected $3$-thread, take a two-element control set $C_i$ as in Lemma~\ref{lem:three-thread}, using a common witness whenever endpoints coincide. Let $N=k-2$.

If $|R|\ge N+1$, the $N$ control sets have at most $2N$ incidences
with colors of $R$, so some $\alpha\in R$ belongs to at most one
control set.  The same conclusion holds if $|R|=N$ but some
$\alpha\in R$ has incidence at most one.  Set $\col(u_0)=\alpha$.
Then at most one thread is rigid.  The two colors on $a,b$, together
with the first color on a possible rigid thread, give at most three
fixed neighbor colors, while at least one thread is flexible. Lemma~\ref{lem:singleton} chooses the first colors on the flexible threads so that $u_0$ is conflict-free. All these colors differ from $\alpha$.

Otherwise, $|R|=N$ and every color of $R$ lies in at least two control sets. Since there are exactly $N$ two-element control sets, each is contained in $R$, and every color of $R$ has incidence exactly two. Equality in $|R|\ge k+2-4=k-2$ also implies that $\col(a),c_a,\col(b),c_b$ are pairwise distinct members of $L(u_0)$. In particular, $\col(a)\ne\col(b)$ and $\col(a),\col(b)\notin R$. Choose any $\alpha\in R$ for $u_0$ and, on each $3$-thread, a first color in $C_i\setminus\{\alpha\}$. These first colors lie in $R$, so $\col(a)$ occurs exactly once in $N_G(u_0)$.

In both cases, Lemmas~\ref{lem:three-thread} and~\ref{lem:shared-boundary} complete the selected threads. Since $\alpha\in R$, restoring $u_0$ is proper at $a,b$ and preserves their witnesses. The resulting PCF $L$-coloring of $G$ is a contradiction. Thus (R5) holds.

\medskip
\noindent\textbf{(R6).}
Suppose that a $k$-vertex $u_0$, $k\ge4$, has exactly $k-1$
$2$-neighbors and that at least $k-3$ of them lie on maximal $3$-threads.
Select such neighbors $u_1,\ldots,u_{k-3}$, label the two remaining
$2$-neighbors $u_{k-2},u_{k-1}$, and write the selected $3$-threads as
\[
 u_0-u_i-v_i-w_i-x_i,
 \qquad 1\le i\le k-3.
\]
Let $u_k$ be the unique neighbor of $u_0$ that is not a $2$-vertex, and for $i\in\{k-2,k-1\}$ let $y_i$ be the neighbor of $u_i$ other than $u_0$. Define
\[
 D=\{u_0,u_1,\ldots,u_{k-1}\}\cup
   \{v_i,w_i:1\le i\le k-3\},
 \qquad H=G-D.
\]

By (R1), $d_G(u_k)\ge3$, whereas every vertex of $D\setminus\{u_0\}$ has degree $2$. Thus $u_k\notin D$. An edge from $u_k$ to $u_i,v_i$, or $w_i$ on a selected $3$-thread would form a cycle of length at most $5$ through $u_0$; an edge to $u_{k-2}$ or $u_{k-1}$ would form a triangle. Hence $u_k$ loses only $u_0$, and $d_H(u_k)\ge2$.

Fix $i\in\{k-2,k-1\}$, and let $j$ be the other index in this set. The vertex $y_i$ is outside $D$: equality with $u_h,v_h$, or $w_h$ on a selected $3$-thread gives a cycle of length $3,4$, or $5$, respectively, and equality with $u_j$ gives a triangle. The vertices $y_{k-2},y_{k-1}$ are distinct from each other, from $u_k$, and from every $x_h$, since these identifications give cycles of length at most $6$. Also $x_h\ne u_k$, since otherwise the path $u_0-u_h-v_h-w_h-x_h$ and the edge $u_0u_k$ give a $5$-cycle. If $y_i$ had a neighbor in $D$ other than $u_i$, the edge to that neighbor, together with $y_i u_i u_0$ and the relevant path from $u_0$, would give a cycle of length at most $6$. Thus $y_i$ loses only $u_i$ and is non-isolated by (R1).

Each $x_h$ is a $4^+$-vertex by (R2)--(R3). By (R4), it has a $3^+$-neighbor different from $u_0$, since $x_hu_0$ would create a $5$-cycle. This neighbor is outside $D$, so $x_h$ is non-isolated in $H$. Color $H$ by minimality. Fix witnesses at $u_k,y_{k-2},y_{k-1}$ and at each distinct $x_h$, using one common witness when several $x_h$ coincide.

Define
\[
 L_0=L(u_0)\setminus
 \{\col(u_k),c_{u_k},\col(y_{k-2}),\col(y_{k-1})\},
\]
so $|L_0|\ge k-2$.  For $1\le i\le k-3$, construct a two-element
control set $C_i$ for the $3$-thread as in
Lemma~\ref{lem:three-thread}.  For $i=k-2,k-1$, put
\[
 C_i\subseteq
 L(u_i)\setminus\{\col(y_i),c_{y_i}\},
 \qquad |C_i|=2.
\]

Fix $\alpha\in L_0$ and let
\[
 U_\alpha=\{u_i: \alpha\in C_i,
                 1\le i\le k-1\}.
\]
If $|U_\alpha|\ge3$, color every vertex of $U_\alpha$ with
$\alpha$, and for each remaining $u_i$ choose a color in
$C_i\setminus\{\col(u_k)\}$.  At most $k-3$ distinct colors now
occur on $u_1,\ldots,u_{k-1}$, whereas $|L_0|\ge k-2$; choose
\[
 \col(u_0)\in
 L_0\setminus\{\col(u_1),\ldots,\col(u_{k-1})\}.
\]
No $u_i$ has color $\col(u_k)$, so $\col(u_k)$ is unique in
$N_G(u_0)$.  For each $i\le k-3$, we have $\col(u_i)\in C_i$ and $\col(u_i)\ne\col(u_0)$, as required by Lemma~\ref{lem:three-thread}.

Now suppose $|U_\alpha|\le2$. Set $\col(u_0)=\alpha$, and for each $u_i\in U_\alpha$ choose $\col(u_i)\in C_i\setminus\{\alpha\}$. At most three neighbors of $u_0$ are now colored: $u_k$ and the vertices in $U_\alpha$. At least one of $u_1,\ldots,u_{k-1}$ remains uncolored, since $k\ge4$. For every such vertex $u_i$, we have $\alpha\notin C_i$, so both colors of $C_i$ are available. Apply Lemma~\ref{lem:singleton} to choose their colors. Then $u_0$ is conflict-free and each $u_i$ avoids $\alpha$.

Complete the selected $3$-threads by Lemmas~\ref{lem:three-thread} and~\ref{lem:shared-boundary}. For $i=k-2,k-1$, the chosen color of
$u_i$ avoids both $\col(y_i)$ and $c_{y_i}$, while
$\col(u_0)\ne\col(y_i)$ by the definition of $L_0$; hence the
coloring is proper, $u_i$ is conflict-free, and the witness at $y_i$
is preserved.  This gives a proper conflict-free coloring of $G$, a
contradiction.  Thus (R6) holds.

\medskip
\noindent\textbf{(R7).}
Let $u_0$ be a $3$-vertex with two $2$-neighbors $u_1,u_2$, and let
$u_3$ be its third neighbor.  For $i=1,2$, let $x_i$ be the other
neighbor of $u_i$.  By (R2), $d_G(x_i)\ge3$.  Suppose, by symmetry,
that $d_G(x_1)=3$. Let $H=G-\{u_0,u_1,u_2\}$.

By girth at least $7$, the vertices $u_3,x_1,x_2$ are pairwise
distinct.  By (R4), $u_3$ is a $3^+$-vertex, since it is the only
neighbor of $u_0$ that is not a $2$-vertex.  Moreover, $u_3$ is adjacent to neither $u_1$ nor $u_2$, and $x_i$ is adjacent to neither $u_0$ nor $u_{3-i}$: any such edge would create a cycle of length at most $4$.  Hence each of $u_3,x_1,x_2$ loses exactly one neighbor when $u_0,u_1,u_2$ are deleted.  Since $d_G(u_3),d_G(x_1),d_G(x_2)\ge3$, all three vertices are non-isolated in $H$. Color $H$ by minimality.  Fix witnesses $c_{u_3}$ and
$c_{x_2}$.

Choose
\[
 \col(u_0)\in
 L(u_0)\setminus
 \{\col(u_3),c_{u_3},\col(x_1),\col(x_2)\},
\]
then
\[
 \col(u_2)\in
 L(u_2)\setminus
 \{\col(u_0),\col(x_2),c_{x_2}\},
\]
and finally
\[
 \col(u_1)\in
 L(u_1)\setminus
 \{\col(u_0),\col(x_1),\col(u_2)\}.
\]
All choices exist.  The vertices $u_1,u_2$ are conflict-free because
$\col(u_0)$ was chosen different from $\col(x_1)$ and
$\col(x_2)$.  The witness at $u_3$ and the witness at $x_2$ are
preserved.  The two neighbors of $x_1$ already present in $H$ have
distinct colors, since $x_1$ has degree $2$ in $H$ and is
conflict-free there; after $u_1$ is restored, Lemma~\ref{lem:three-neighbors}
applies at $x_1$.  Finally $\col(u_1)\ne\col(u_2)$, so the same
lemma applies at $u_0$.  This contradiction proves (R7).
\end{proof}

We shall also use the following separation consequence of the girth
assumption.

\begin{lemma}[Branch separation]\label{lem:separation}
Let $v$ be a $3^+$-vertex of $G$. Distinct maximal threads incident with $v$ have distinct neighbors of $v$ and disjoint sets of internal vertices. If a maximal $s$-thread and a maximal $t$-thread share their other endpoint, then $(s,t)$ is $(2,3)$, $(3,2)$, or $(3,3)$. No other endpoint of such a thread is adjacent to $v$.
\end{lemma}
\begin{proof}
Once the neighbor of $v$ on a thread is fixed, all subsequent edges through $2$-vertices are determined. Thus two distinct maximal threads cannot have the same neighbor of $v$ or share an internal vertex. By (R3), $s,t\in\{1,2,3\}$. If their other endpoints coincide, the two paths form a cycle of length $s+t+2$. Since the girth is at least $7$, the stated pairs are the only possibilities. An edge from $v$ to the other endpoint of a maximal $t$-thread would form a cycle of length $t+2\le5$.
\end{proof}

\begin{lemma}[Non-isolated boundary vertices]\label{lem:boundary-survival}
Let $v$ be a $4^+$-vertex, and let $H$ be obtained by deleting $v$ and all $2$-vertices on maximal threads incident with $v$. Every $3^+$-neighbor of $v$ and every other endpoint of an incident maximal thread is non-isolated in $H$.
\end{lemma}
\begin{proof}
Let $D$ denote the deleted set. Each vertex of $D\setminus\{v\}$ has degree $2$, with both neighbors specified by its thread. Therefore, an edge from a vertex outside $D$ to a vertex of $D\setminus\{v\}$ must join an endpoint of a thread to its terminal $2$-vertex.

A direct $3^+$-neighbor $y$ of $v$ cannot be the other endpoint of any thread incident with $v$, by Lemma~\ref{lem:separation}. Hence $y$ loses only the edge $yv$, and therefore has degree at least $2$ after the deletion.

Let $x$ be the other endpoint of a maximal $1$-thread incident with $v$. By Lemma~\ref{lem:separation}, a maximal $1$-thread cannot share its other endpoint with another thread incident with $v$. Thus $x$ loses only the terminal neighbor on its own branch, and $d_G(x)\ge3$, so it remains non-isolated.

Finally let $x$ be the endpoint of a maximal $2$- or $3$-thread incident with $v$. In either case $x$ is a $4^+$-vertex: a $3$-vertex at $x$ would be incident with a $2$-thread, contrary to Proposition~\ref{prop:reductions}\ref{red:3no2}. By Proposition~\ref{prop:reductions}\ref{red:3plus}, $x$ has a $3^+$-neighbor $z$. The vertex $z$ is not $v$, for otherwise the edge $xv$ together with the thread gives a cycle of length at most $5$. Since every vertex of $D\setminus\{v\}$ has degree $2$, the $3^+$-vertex $z$ is not deleted. Hence $x$ remains non-isolated even when several threads end at $x$.
\end{proof}

\section{The weighted thread lemma}\label{sec:weighted}
Let $v$ be a $k$-vertex, $k\ge4$. Define
\begin{align*}
 n_0&:=\#\{\text{maximal $1$-threads from $v$ to a $4^+$-vertex}\},\\
 n_1&:=\#\{\text{maximal $1$-threads from $v$ to a $3$-vertex}\},\\
 n_2&:=\#\{\text{maximal $2$-threads incident with $v$}\},\\
 n_3&:=\#\{\text{maximal $3$-threads incident with $v$}\},
\end{align*}
and let
\[
  r:=k-(n_0+n_1+n_2+n_3).
\]
Thus $r$ is the number of direct $3^+$-neighbors of $v$. Put
\[
  W(v):=n_0+\frac32n_1+2n_2+3n_3.
\]

\begin{lemma}[Weighted thread lemma]\label{lem:weighted}
For every $k$-vertex $v$ with $k\ge4$ in the minimum counterexample,
\[
  W(v)\le 3k-8.
\]
\end{lemma}
\begin{proof}
Suppose $W(v)>3k-8$. A direct calculation gives
\begin{equation}\label{eq:deficit}
  3k-8-W(v)=3r-8+2n_0+\frac32n_1+n_2.
\end{equation}
By Proposition~\ref{prop:reductions}\ref{red:3plus}, the root $v$ has a direct $3^+$-neighbor, so $r\ge1$.  If $r\ge3$, the right-hand side is at least $1$, a contradiction. Hence $r\in\{1,2\}$.

Delete $v$ and all internal $2$-vertices on every maximal thread incident with $v$, and call the remaining graph $H$.  By Lemma~\ref{lem:boundary-survival}, every $3^+$-neighbor of $v$ and every other endpoint of an incident thread is non-isolated in $H$.  Color $H$ by minimality.  For every direct $3^+$-neighbor $y$ of $v$, fix a witness $c_y\in\UU_\col(y,H)$; for every boundary vertex $x$ of a thread that needs a witness, fix one witness $c_x$. If several maximal $2$- or $3$-threads share the same boundary vertex, use the same $c_x$ for all of them.

Set
\[
\begin{aligned}
 R:=L(v)\setminus\Bigl(&
 \{\col(y),c_y:y\text{ direct }3^+\text{-neighbor}\}\\
 &\cup\{\col(x):x\text{ is the other endpoint of a maximal $1$-thread}\}
 \Bigr).
\end{aligned}
\]
Hence
\begin{equation}\label{eq:rootpalette}
  |R|\ge k+2-2r-n_0-n_1.
\end{equation}
Let $N:=n_2+n_3$. By Lemmas~\ref{lem:two-thread} and~\ref{lem:three-thread}, each of the $N$ long branches has a two-element control set.

\medskip\noindent\textbf{Case 1: $r=1$.}
By Proposition~\ref{prop:reductions}\ref{red:saturated}, $n_3\le k-4$, and hence
\begin{equation}\label{eq:three-short}
  n_0+n_1+n_2\ge3.
\end{equation}
Since \eqref{eq:deficit} is negative,
\[
  2n_0+\frac32n_1+n_2<5.
\]
Together with \eqref{eq:three-short}, this implies $n_0\le1$.

From \eqref{eq:rootpalette},
\[
 |R|\ge k-n_0-n_1=n_2+n_3+1=N+1.
\]
The $N$ control sets contribute at most $2N$ incidences with colors of $R$, so some $\alpha\in R$ belongs to at most one control set. Set $\col(v)=\alpha$.  Because $\alpha\in R$, restoring $v$ is proper at every direct $3^+$-neighbor and preserves the fixed witness there.  At most one long branch is rigid.

If $n_0=1$, choose a color for the first vertex of that branch by Lemma~\ref{lem:one-four}. Also choose the first color of the rigid long branch, if one exists. Together with the unique direct neighbor, these give at most three fixed neighbor colors. Every remaining branch offers at least two first colors, by Lemmas~\ref{lem:one-three}--\ref{lem:three-thread}, and at least one such branch exists by \eqref{eq:three-short}. Apply Lemma~\ref{lem:singleton} to make $v$ conflict-free, then complete the branches by the extension lemmas and Lemma~\ref{lem:shared-boundary}. This is a contradiction.

\medskip\noindent\textbf{Case 2: $r=2$.}
By Proposition~\ref{prop:reductions}\ref{red:n3}, $n_3\le k-3$. Since $n_0+n_1+n_2+n_3=k-2$, we have $n_0+n_1+n_2\ge1$. Negativity in \eqref{eq:deficit} gives
\[
  2n_0+\frac32n_1+n_2<2.
\]
Thus the only possibilities are
\begin{equation}\label{eq:A1A2}
  (n_0,n_1,n_2)=(0,0,1)
  \quad\text{or}\quad
  (n_0,n_1,n_2)=(0,1,0).
\end{equation}
Call these $A_1$ and $A_2$, respectively. For $A_1$, $N=k-2$; for $A_2$, $N=k-3$. In both cases \eqref{eq:rootpalette} gives $|R|\ge N$.

If $|R|\ge N+1$, the incidence count used in Case~1 gives a color $\alpha\in R$ belonging to at most one control set. Whenever such a color exists, set $\col(v)=\alpha$. There is at most one rigid long branch. Fix its first color, if it exists, together with the colors on the two $3^+$-neighbors of $v$. These give at most three fixed neighbor colors. Since $n_0=0$ and $k-2\ge2$, at least one branch remains flexible. Lemma~\ref{lem:singleton} chooses colors on the remaining neighbors to make $v$ conflict-free, and the extension lemmas complete the coloring. We may therefore assume the tight situation
\begin{equation}\label{eq:tight}
  |R|=N
  \quad\text{and every color of $R$ belongs to at least two control sets}.
\end{equation}
Write the $N$ two-element control sets as $Q_1,\ldots,Q_N$. Their total incidence with $R$ is at most $2N$, so \eqref{eq:tight} forces
\begin{equation}\label{eq:control-contained}
  Q_i\subseteq R\quad\text{for every $i$},
  \qquad
  \text{and every color of $R$ has incidence exactly $2$}.
\end{equation}

Let $p,q$ be the colors of the two direct neighbors of $v$. Tightness forces
\begin{equation}\label{eq:pq}
  p\ne q,
  \qquad
  p,q\notin R.
\end{equation}
In $A_1$ the estimate \eqref{eq:rootpalette} excludes four displayed colors from $L(v)$, and in $A_2$ it excludes five. Since $|L(v)|=k+2$ and $|R|=N$, equality holds in this estimate. All the displayed excluded colors are therefore pairwise distinct members of $L(v)$, proving \eqref{eq:pq}.

\smallskip\noindent\emph{Subcase $A_2$.}
Choose any $\alpha\in R$ and set $\col(v)=\alpha$.  All long branches are $3$-threads.  For each such branch choose its first color in $C_i\setminus\{\alpha\}$; since $C_i\subseteq R$ by \eqref{eq:control-contained} and $p\notin R$ by \eqref{eq:pq}, this color differs from $p$. The unique $n_1$-branch has at least two admissible first colors relative to the root color $\alpha$, so choose one different from $p$. Since $q\ne p$, the color $p$ occurs exactly once in $N(v)$.

\smallskip\noindent\emph{Subcase $A_1$ with $N\ge3$.}
Let $B_0$ be the control set of the unique $2$-thread. Since $|B_0|=2<|R|$, choose $\alpha\in R\setminus B_0$. The $2$-thread is flexible, so choose its first color different from $p$. For every $3$-thread choose its first color in $C_i\setminus\{\alpha\}$; by \eqref{eq:control-contained} this chosen color lies in $R$ and hence avoids $p$. Again $p$ occurs exactly once in $N(v)$.

\smallskip\noindent\emph{Subcase $A_1$ with $N=2$.}
Then $k=4$ and $(n_0,n_1,n_2,n_3)=(0,0,1,1)$. Write $R=\{\alpha,\beta\}$. By \eqref{eq:control-contained}, both control sets equal $R$. Set $\col(v)=\alpha$. On the $3$-thread, color its neighbor of $v$ with $\beta$, as allowed by Lemma~\ref{lem:three-thread}. Write the $2$-thread as $v-u-w-x$ and put $s=\col(x)$. Its control set is $B=R=\{\alpha,\beta\}$. In the case $\alpha\in B$ of Lemma~\ref{lem:two-thread}, the available colors for $u$ are
\[
 F_\alpha=L(u)\setminus\{\alpha,s,\beta\}.
\]
Since $|L(u)|=4$, choose $a\in F_\alpha$ and set $\col(u)=a$ and $\col(w)=\beta$. In particular, $a\ne\beta$. The four neighbors of $v$ have colors $p,q,a,\beta$. As $p,q\notin R$, the color $\beta$ occurs exactly once in $N(v)$.

In every subcase, complete the branches by Lemmas~\ref{lem:one-four}--\ref{lem:three-thread} and~\ref{lem:shared-boundary}. The resulting PCF $L$-coloring of $G$ is a contradiction. Therefore $W(v)\le3k-8$.
\end{proof}

\section{Discharging and proof of the main theorem}\label{sec:discharging}
\begin{proof}[Proof of Theorem~\ref{thm:main}]
Give every vertex $v$ the initial charge
\[
  \mu(v)=3d(v)-8.
\]
Since $\mad(G)<8/3$,
\[
  \sum_{v\in V(G)}\mu(v)=6|E(G)|-8|V(G)|<0.
\]
Redistribute charge as follows.
\begin{enumerate}[label=\textnormal{(D\arabic*)}]
\item\label{D:3} A $3$-vertex with exactly one $2$-neighbor sends $1$ to that neighbor. A $3$-vertex with exactly two $2$-neighbors sends $1/2$ to each.
\item\label{D:long} For every maximal $2$- or $3$-thread with $4^+$ endvertices, each endvertex sends $1$ to every internal $2$-vertex of the thread.
\item\label{D:one} If $vux$ is a maximal $1$-thread with $d(v)\ge4$, then $v$ sends $1$ to $u$ when $d(x)\ge4$, and sends $3/2$ to $u$ when $d(x)=3$.
\end{enumerate}
Let $\mu^*(v)$ be the final charge.

We show that $\mu^*(v)\ge0$ for every vertex $v$. By (R1), it suffices to consider vertices of degree $2$, $3$, and at least $4$.

A $2$-vertex starts with charge $-2$. By Proposition~\ref{prop:reductions}\ref{red:no4}, it belongs to a maximal thread of length at most $3$. If it lies on a maximal $2$- or $3$-thread, both endvertices are $4^+$-vertices by Proposition~\ref{prop:reductions}\ref{red:3no2}, and each sends $1$ by rule~\ref{D:long}. On a $1$-thread, two $4^+$-endpoints send $1$ each, while a $4^+$-endpoint and a $3$-endpoint send at least $3/2+1/2$. If both endpoints are $3$-vertices, Proposition~\ref{prop:reductions}\ref{red:3two} implies that each has this vertex as its unique $2$-neighbor, so each sends $1$. Thus every $2$-vertex receives at least $2$.

A $3$-vertex starts with charge $1$. By Proposition~\ref{prop:reductions}\ref{red:3plus}, it has at most two $2$-neighbors, so rule~\ref{D:3} makes it send at most $1$.

Finally, let $v$ be a $k$-vertex with $k\ge4$. Rules~\ref{D:long} and~\ref{D:one} make $v$ send exactly
\[
  n_0+\frac32n_1+2n_2+3n_3=W(v).
\]
Lemma~\ref{lem:weighted} gives
\[
  \mu^*(v)=3k-8-W(v)\ge0.
\]

Every vertex therefore has nonnegative final charge. Since charge is only redistributed,
\[
  0\le\sum_{v\in V(G)}\mu^*(v)=\sum_{v\in V(G)}\mu(v)<0,
\]
a contradiction.
\end{proof}

\section{Concluding remarks}
For planar graphs of girth at least $7$, Euler's formula gives only $\mad(G)<14/5$, which does not meet the bound in Theorem~\ref{thm:main}. Thus the cases of girth $7$ and $6$, the latter asked by Wang and Zhang~\cite[Problem~5.4]{WangZhang2025}, are not resolved by the present result.

\section*{Acknowledgment}
This work was supported by the National Natural Science Foundation of China [grant number 12471330] and the Shandong Provincial Natural Science Foundation [grant number ZR2025MS71].

\section*{Declaration of competing interest}
The authors declare that they have no known competing financial interests or personal relationships that could have appeared to influence the work reported in this paper.

\section*{Declaration of generative AI and AI-assisted technologies in the manuscript preparation process}
During the preparation of this work, the authors used OpenAI ChatGPT to assist with language editing, organization of the exposition, literature-search support, and exploratory checking of proof exposition. After using this tool, the authors reviewed and edited the content as needed and take full responsibility for the content of the article.

\end{document}